\documentclass{article}
\usepackage{amsmath, amssymb, amsthm, amscd, amsfonts, eucal, hyperref,etoolbox}
\usepackage[all]{xy}

\newtheorem{theorem}{Theorem}[section]

\newtheorem{lemma}[theorem]{Lemma}
\newtheorem {corollary}[theorem]{Corollary}
\theoremstyle {definition}

\theoremstyle {remark}
\newtheorem{remark}[theorem]{Remark}

\def\exp{\operatorname{exp}}

\def\max{\operatorname{max}}

\makeatletter
\newcommand{\subjclass}[2][2010]{%
  \let\@oldtitle\@title%
  \gdef\@title{\@oldtitle\footnotetext{#1 \emph{Mathematics subject classification.} #2}}%
}
\newcommand{\keywords}[1]{%
  \let\@@oldtitle\@title%
  \gdef\@title{\@@oldtitle\footnotetext{\emph{Key words and phrases.} #1.}}%
}
\makeatother

\author{NGUYEN TAT THANG\\
Institute of Mathematics\\
 18 Hoang Quoc Viet road\\
Cau Giay District, 10307 Hanoi, Vietnam\\
\texttt{ntthang@math.ac.vn}}
\title{Generalized Broughton polynomials and characteristic varieties}
\date{}
\begin{document}

\maketitle

\begin{abstract}
 We introduce a family of generalized Broughton polynomials and compute the characteristic varieties of complement of a curve arrangement defined by fibers of some generalized Broughton polynomials.
\end{abstract}

%%%%%%%%%%%%%%%%%%%%%%%%%%%%%%%%%%%

\section{Introduction}
In \cite{Br} Broughton considered the polynomial
$$f(x, y)= x(xy-1).$$
The associated function $f: \mathbb{C}^2\to \Bbb{C}$ has no critical value, but the fiber $f^{-1}(0)$ is not diffeomorphic to the generic one. This is explained by the existence of the so-called  "critical value at infinity", see \cite{HL}, \cite{Br}, \cite{D1}. 

In the paper \cite{Z} Zahid introduced a family of polynomials: 
$$f_{p,q}(x, y)= x^p[xy(x + 2) \cdots (x + q) - 1],$$
which are called {\it generalized Broughton polynomials}, where $p\geq 1$ and $q\geq 1$ are integer number, with the convention 
$$f_{p,1}= x^p (xy -1).$$
By computing the characteristic variety $\mathcal{V}_1(M),$ where 
$$M= \mathbb{C}^2\setminus (C_0\cup C_1)$$ 
is a complement of a curve arrangement defined by a component of the $0-$fiber:
$$C_0=\begin{cases}
\{xy(x + 2) \cdots (x + q) - 1=0\}\, \, & \textrm{if}\, \, q>1\\
\{xy-1= 0\}\, \, & \textrm{if}\, \, q=1
\end{cases}$$
 and 
the generic fiber of $f_{p, q}$:
$$C_1=\{f_{p,q}(x, y)=1\},$$
 the author obtained examples of characteristic varieties with an arbitrary number of translated components for complements of affine curve arrangements consisting of just two rational curves, see \cite{Z}.

The aim  of this paper is to generalize the Zahid's work in \cite{Z}. More precisely, we introduce a family of {\it generalized Broughton polynomials}, which generalizes the Zahid's one. Namely
$$F(x, y):= p(x) (yq(x)-1)$$
where $p(x), q(x)\in \mathbb{C}[x].$

Put $f(x, y) := F(x, y)-1$ and $g(x, y):= yq(x)-1.$ We denote by $M$ the complement 
$$M= \mathbb{C}^2\setminus \{f(x, y)=0, g(x, y)=0\}.$$
 The main result in this note shows how to compute the characteristic variety  $\mathcal{V}_1(M)$, for all polynomials $p(x), q(x)$ 
 such that they have at least one common root and $p(x) + 1, q(x)$ have not any common root.
 
 In Section 2 we recall the definition and the basic properties of the characteristic and resonance varieties. In Section 3 we compute the characteristic variety  $\mathcal{V}_1(M)$. In particular, we obtain examples of characteristic varieties with an arbitrary number of translated components (Theorem \ref{mainthm}). This is an extension for Theorem 4.1 in \cite{Z}.
%%%%%%%%%%%%%%%%%%%%%%%%%%%%%%%%%

\section{Characteristic and Resonance varieties} 

Let M be a smooth, irreducible, quasi-projective complex variety. The character variety of $M$ is defined by 
$$\mathbb{T}(M) := Hom (H_1(M), \mathbb{C}^*).$$
 This is an algebraic group whose identity irreducible component $\mathbb{T}(M)_1$ is an algebraic torus $(\mathbb{C}^{*})^{b_1(M)}$. Consider the exponential mapping
\begin{align}\label{exp}
\exp : H^1(M, \mathbb{C}) \to H^1(M, \mathbb{C}^{*}) = \mathbb{T}(M)
\end{align}
induced by the usual exponential function $\exp : \mathbb{C}\to \mathbb{C}^{*}$. Clearly
$$\exp (H^1(M, \mathbb{C})) = \mathbb{T}(M)_1.$$

The {\it characteristic varieties} of $M$ are the jumping loci for the first cohomology of $M$, with coefficients in rank one local systems:
$$\mathcal{V}^i_k(M)=\{\rho\in \mathbb{T}(M) : \dim H^i(M, \mathcal{L}_{\rho})\geq k\}.$$
When $i = 1$, we use the simpler notation $\mathcal{V}_k(M) = \mathcal{V}^1_k (M)$.

Fundamental results on the structure of the cohomology support loci for local
systems on quasi-projective algebraic varieties were obtained by Beauville \cite{B}, Green
and Lazarsfeld \cite{GL}, Simpson \cite{S} (for the proper case), and Arapura \cite{A} (for the
quasi-projective case and first characteristic varieties $\mathcal{V}_1(M)$).

\begin{theorem}
The strictly positive dimensional irreducible components of the first
characteristic variety $\mathcal{V}_1(M)$ are translated subtori in $\mathbb{T}(M)$ by elements of finite
order. When $M$ is proper, all the components of $\mathcal{V}^i_k(M)$ are translated subtori
in $\mathbb{T}(M)$ by elements of finite order.
\end{theorem}

The strictly positive dimensional irreducible
components of the first characteristic variety $\mathcal{V}_1(M)$ are described as follows.

\begin{theorem}\label{thm1} (\cite{B}, \cite{A})
Let $W$ be a $d-$dimensional irreducible component of $\mathcal{V}_1(M), d>0.$ Then, there is a regular morphism $f: M\to S$ onto a smooth curve $S$ with $b_1(S)=d$ such that the generic fiber $F$ of $f$ is connected, and a torsion character $\rho\in \mathbb{T}(M)$ such that the composition 
$$\pi_1(F)\to \pi_1(M)\to \mathbb{C}^*$$
is trivial and $W= \rho \cdot f^*(\mathbb{T}(S)).$
\end{theorem}

\begin{remark}\label{rm1}
{\rm
 When $M$ is a hypersurface complement in $\mathbb{P}^n$, the curve $S$ in Theorem \ref{thm1} above is obtained
from $\mathbb{C}$ by deleting $d$ points, see \cite{D4}, Theorem 1.11.
}
\end{remark}

If we fix a regular mapping $f : M \to S$ as above, the number of irreducible components $W = \rho \cdot f^*(\mathbb{T}(S))$ obtained by varying the torsion character $\rho$ is given by the following.

\begin{theorem}\label{thm2} (\cite{D3})
For a given regular map $f : M \to S$ as above, the associated
irreducible components $W= \rho \cdot f^*(\mathbb{T}(S))$ are parametrized
by the Pontrjagin dual $\hat{T}(f) = Hom(T(f);\mathbb{C}^*)$ of the finite abelian group
$$T(f) = \frac{ker \{f^* : H_1(M) \to H_1(S)\}}{im \{i^* : H_1(F)\to H_1(M)\}}$$
if $\chi(S) < 0$ and by the non-trivial elements of this Pontrjagin dual $\hat{T}(f)$ if $\chi(S)=0$.
\end{theorem}

The group $T(f)$ is determined as follows.

\begin{theorem}(\cite{D3}) \label{thm5}
Let $S$ is a non-proper smooth curve and $f: M\to S$ be a regular function. Then the group $T(f)$ is computed by the following
$$T(f)= \oplus_{c\in C(h)}\mathbb{Z}/m_c\mathbb{Z},$$
where $m_c$ is the multiplicity of the divisor $f^{-1}(c)$ and $C(f)$ is the set of bifurcation values of $f$.
\end{theorem}

The (first) {\it resonance varieties} of $M$ are the jumping loci for the first
cohomology of the complex $H^*(H^*(M, \mathbb{C}), \alpha \wedge),$ namely
$$\mathcal{R}_k(M)=\{\alpha\in H^1(M, \mathbb{C}): \dim H^1(H^*(M, \mathbb{C}), \alpha \wedge)\geq k\}.$$

The relation between the resonance and characteristic varieties can be
summarized as follows, see \cite{D5}.

\begin{theorem}\label{thm3}
Assume that $M$ is any hypersurface complement in $\mathbb{P}^n$.
Then the irreducible components $E$ of the resonance variety $\mathcal{R}_1(M)$ are linear
subspaces in $H^1(M, \mathbb{C})$ and the exponential mapping (\ref{exp}) sends these
irreducible components $E$ onto the irreducible components $W$ of $\mathcal{R}_1(M)$ with
$1\in W$.
\end{theorem}
%%%%%%%%%%%%%%%%%%%%%%%%%%%%%%%%%

\section{The Characteristic varieties $\mathcal{V}_1(M)$}

Consider from now on the complement $M = \mathbb{C}^2\setminus C$, where $C =
C_0\cup C_1, C_0 =\{ g(x, y) =0\}$ and $C_1 = \{f(x, y)=0\}$.

By the same argument as in Section 3 in \cite{Z} we can prove the following.

\begin{theorem} The integral (co)homology of the surface $M$ is torsion
free and 
$$b_1(M) = 2, b_2(M) = s+t,$$
 where $s$ and $t$ are the numbers of roots of $q(x)$ and $p(x)q(x)$, respectively. Moreover, the cup-product
$$\cup : H^1(M)\times  H^1(M)\to H^2(M)$$
is non-trivial.
\end{theorem}

Using the definition of the resonance varieties we get the following.

\begin{corollary} The resonance varieties of $M$ are trivial, i.e. $\mathcal{R}_k(M) =
0$ for any $k > 0.$
\end{corollary}

Since the resonance varieties are trivial, and $M$ is a hypersurface complement,
it follows from Theorem \ref{thm3} that the characteristic varieties $\mathcal{V}_1(M)$
can contain only isolated points and 1-dimensional translated components.
In this section we determine the latter ones.

In view of Theorem \ref{thm1} and Remark \ref{rm1}, any such component
comes from a mapping $h: M\to \mathbb{C}^*$. 
If we regard $h$ as a regular function on
the affine variety $M$, it follows that $h$ should have the form
$$h = \frac{P(x, y)}{f^m g^n}$$
for some polynomial $P$ and some positive integers $m, n$. If $P$ is not in the
multiplicative system spanned by $f$ and $g$, then $P$ vanishes at some point
of $M$ and this is a contradiction. It follows that we may assume that
$$h = f^m g^n$$
for some (positive or negative) integers $m, n$. Now, we are looking for all such mappings such that they have multiple fibers and connected generic fiber. 

\begin{lemma}\label{lmtinhlienthong}
For all integer numbers $m> 1, n>1$ and $c\in\mathbb{C}\setminus \{0\}$, then the generic fiber of the polynomial $f^m(x, y) + cg^n(x, y)$ is connected.
\end{lemma}

We need the following fact.
\begin{lemma}(\cite{D3})\label{lmtieuchuan} For any polynomial map $P: \mathbb{C}^n\to \mathbb{C}$ the followings are equivalent:
\begin{enumerate}
\item[(1)] The generic fiber of $P$ is connected;
\item[(2)] There do not exist polynomials $H: \mathbb{C}\to \mathbb{C}$ and $Q: \mathbb{C}^n\to \mathbb{C}$ such that $\deg(H)> 1$ and $P= H(Q).$
\end{enumerate}
\end{lemma}

\begin{proof}[Proof of Lemma \ref{lmtinhlienthong}]
Let $\Phi: \mathbb{C}^2\to \mathbb{C}^2$ be given by $\Phi(x, y) = (x, g(x, y)).$ We have 
$$f^m(x, y) + cg^n(x, y)= h\circ \Phi,$$
where $h(u, v):= (p(u)v - 1)^m + cv^n$. 

It is easy to see that the restriction of $\Phi $ on $\mathbb{C}^2\setminus A$ is a homeomorphism, where $A= \{(a, y): q(a) = 0, y\in \mathbb{C}\}.$ Then, the generic fiber of $f^m(x, y) + cg^n(x, y)$ is connected if and only if the generic fiber of $h(u, v)$ is connected. 

Now, we assume by contradiction that the generic fiber of $h(u, v)$ is not connected. According to Lemma \ref{lmtieuchuan}, there are polynomials $H: \mathbb{C}\to \mathbb{C}$ and $Q: \mathbb{C}^2\to \mathbb{C}$ such that $\deg(H)> 1$ and 
$$(p(u)v - 1)^m + cv^n= H(Q(u, v)).$$
We consider the singular locus of the polynomials in the above equality. Since $\deg(H)>1$ then the singular locus of $H(Q(u, v))$ has dimension at least one. In particular, there are infinitely many points. However, singular points of $h(u, v)$ are roots of the following systems.
$$\begin{cases}
p^{'}(u)=0, \\
mp(u)(p(u)v-1)^{m-1}+cnv^{n-1}=0
\end{cases}$$
or
$$\begin{cases}
p(u)=0.\\
v=0
\end{cases}$$
It is easy to see that the above systems have only finitely many points. Contradiction.
\end{proof}

\begin{lemma}\label{lm3} Assume that the map $h= f^mg^n: M\to \mathbb{C}^*$ has connected generic fiber and a multiple fiber. Then $n=0$ and $m= \pm 1$.
\end{lemma}

\begin{proof} If $n=0$ then $m= \pm 1$, because $h$ has connected generic fiber. Similarly, if $m=0$ then $n= \pm 1.$ However, since $\deg(f)> \deg(g)$, it is easy to show that the function $g: \mathbb{C}^2\setminus \{fg=0\}\to \mathbb{C}^{*}$ has not any multiple fiber. 

Now, we assume that $mn\neq 0$. Since $M= \mathbb{C}^2\setminus \{fg=0\}$ and $f, g$ are two irreducible polynomial then the map $h: M\to \mathbb{C}^{*}$ has multiple fiber if and only if, there exist $c\in \mathbb{C}^{*}, h_1\in \mathbb{C}[x, y], h_1\not\vdots f, h_1\not\vdots g$ and integer numbers $s, l, k, |s|>1,$ such that
\begin{equation}
f^mg^n= c+ h_1^sf^lg^k.
\end{equation}
Since $f, g, h_1$ are pairwise relatively prime then $ml\geq 0$ and $nk\geq 0.$ There are four cases.

a) $m, l, n, k\geq 0$: This implies that $l = k =0$ and the generic fiber of $h$ has at least $|s|>1$ connected components. Contradiction.

b) $m, l, n, k\leq 0$: By dividing two sides of the equality (2) by the lowest powers of $f$ and $g$, one can prove that $m=l$ and $n=k.$ It means $$(f^{m}g^{n})^{-1}= \frac{1}{c}(1-h_1^s).$$
So the generic fiber of $f^mg^n$ is not connected.

c) $m, l\geq 0$ and $n, k\leq 0$: Similarly, we get $l=0$ and $n=k.$ Hence $f^m= cg^{-n}+ h_1^s$. But, this implies that the generic fiber of the polynomial $f^m- cg^{-n}$ is not connected, contradicts to Lemma \ref{lmtinhlienthong}.

d) $m, l\leq 0$ and $n, k\geq 0$: By the same argument, we also obtain the contradiction.
\end{proof}

The main result in this paper is the following.

\begin{theorem}\label{mainthm}
Let $p(x)$ and $q(x)\in \mathbb{C}[x]$ be two polynomials such that they have at least one common root and $p(x) + 1, q(x)$ have no common root. Then, if there exist an integer number $s>1$ and a polynomial $p_1\in \mathbb{C}[x]$ such that 
$$p(x)= p_1(x)^s,$$
 the strictly positive dimensional components of $\mathcal{V}_1(M)$ are the translated 1-dimensional sub-tori
$$W_j = \epsilon_j \times \mathbb{C}^{*},$$
 where $d$ is the maximum of the exponent $s$ above and $\epsilon_j = \exp{(2\pi ij/d)}$ for $j = 1, 2, \ldots, d-1$. Moreover, 
for a local system $\mathcal{L}\in W_j$ one has $\dim H^1(M, \mathcal{L}) \geq  1$ and equality holds
with finitely many exceptions.

Otherwise, there do not exist strictly positive dimensional components of $\mathcal{V}_1(M)$.

\end{theorem}

\begin{proof}
According to Theorem \ref{thm1} and Remark \ref{rm1}, any translated positive dimensional component of $\mathcal{V}_1(M)$ 
comes from a map $h: M\to \mathbb{C}^*$ which has connected generic fiber.
 
According to Lemma \ref{lm3}, the only morphisms associated to strictly positive dimensional components of $\mathcal{V}_1(M)$ are $f: M\to \mathbb{C}^{*}$ and $f^{-1}: M\to \mathbb{C}^{*}, z\mapsto f(z)^{-1},$ but they give the same associated component of $\mathcal{V}_1(M)$. Thus all translated positive dimensional components of $\mathcal{V}_1(M)$ are associated to the map $f: M\to \mathbb{C}^{*}$.

On the other hand, it is easy to see that the only possibly multiple fiber of $f$ is $f^{-1}(-1)$. Hence, according to Theorem \ref{thm5}, if $p(x)$ is not a power of a polynomial then $T(f)=0$ and there does not exist strictly positive dimensional components of $\mathcal{V}_1(M)$; unless $T(f)= \mathbb{Z}/ d\mathbb{Z}$, where 
$$d =\max\{s\in \mathbb{N}: p(x)= p_1(x)^s, p_1\in \mathbb{C}[x]\}.$$
We now consider the later case. It is deduced from Theorem \ref{thm3} that there are exactly $d-1$ associated 1-dimensional translated components. If we identify $\mathbb{T}(M) = \mathbb{C}^{*}$ by associating to a local system $\mathcal{L}\in \mathbb{T}(M)$ the
two monodromies $(\lambda_0, \lambda_1)$ about the curves $C_0$ and $C_1$, and in a similar way
$\mathbb{T}(\mathbb{C}^{*}) = \mathbb{C}^{*}$, then the induced morphism
$$f^{*} : \mathbb{T}(\mathbb{C}^{*})\to  \mathbb{T}(M)$$
is just  $\lambda\mapsto (1, \lambda)$.

With these identifications, the above $d-1$ associated 1-dimensional translated components of $\mathcal{V}_1(M)$ are given by $W_j = \epsilon_j \times \mathbb{C}^{*}$, where $\epsilon_j = \exp{(2\pi ij/d)}$ for $j = 1, 2, \ldots, d-1$. 

The inequality on dimension of cohomology group of $M$ is the direct consequence  of Corollary 5.9 in \cite{D3}.
\end{proof}

%\section*{Acknowledgments} The author would like to thank Professor Alexandru Dimca for useful discussions. This work was supported by NAFOSTED (Vietnam) and EMMA (Eramus Mundus).

\end{document}